\documentclass[12pt]{amsart}

\usepackage[colorlinks=true,linkcolor=blue]{hyperref}
\usepackage{amsmath}
\usepackage{amssymb}
\usepackage{amsthm}

\usepackage{graphicx}
\usepackage{epstopdf}
\usepackage{epsfig}

\usepackage{tikz}
\usepackage{tikz-cd}

\theoremstyle{plain} 
\newtheorem{thm}{Theorem}[section]
\newtheorem{prop}[thm]{Proposition}
\newtheorem{lemma}[thm]{Lemma}
 
\newtheorem{question}[thm]{Question}

\theoremstyle{remark}
\newtheorem{remark}[thm]{Remark}

\theoremstyle{definition}
\newtheorem{defin}[thm]{Definition}

\newcommand{\CC}{\mathbb{C}}

\newcommand{\PP}{\mathbb{P}}
\newcommand{\QQ}{\mathbb{Q}}

\newcommand{\ZZ}{\mathbb{Z}}

\newcommand{\kbar}{\overline{k}}
\newcommand{\lra}{\longrightarrow}

\DeclareMathOperator{\Spec}{Spec}

\DeclareMathOperator{\Gal}{Gal}

\DeclareMathOperator{\charact}{char}

\DeclareMathOperator{\Aut}{Aut}

\newcommand{\fm}{ {\mathfrak m} }
\newcommand{\fn}{ {\mathfrak n} }
\newcommand{\fo}{ {\mathfrak o} }
\newcommand{\fp}{ {\mathfrak p} }
\newcommand{\fq}{ {\mathfrak q} }

\begin{document}

\title{Specializations of Iterated Galois Groups of PCF Rational Functions}

\author{Robert L. Benedetto}
\address{Department of Mathematics and Statistics \\ Amherst College \\ P. O. Box 5000 \\ MA 01002-5000 \\ USA }
\email{rlbenedetto@amherst.edu}

\author{Dragos Ghioca}
\address{Department of Mathematics \\ University of British Columbia \\ 1984 Mathematics Road \\ BC V6T 1Z2 \\ Canada }
\email{dghioca@math.ubc.ca}

\author{Jamie Juul}
\address{Department of Mathematics \\ 1874 Campus Delivery \\ Fort Collins \\ CO 80523-1874 \\ USA }
\email{jamie.juul@colostate.edu}

\author{Thomas J. Tucker}
\address{Department of Mathematics \\ Hylan Building \\ University
of Rochester \\ Rochester, NY 14627 \\ USA }
\email{thomas.tucker@rochester.edu}

\subjclass[2020]{Primary: 37P05; Secondary: 11G50, 14G25}
\keywords{arboreal representation, Frattini group}

\begin{abstract}
We obtain a criterion for when the specialization of the iterated Galois group  for a post-critically finite (PCF) rational map is as large as possible, i.e., it equals the generic iterated Galois group for the given map.    
\end{abstract}

\maketitle

\section{Introduction}

\subsection{Notation and setting}
\label{subsec:tec}
We fix the following notation throughout this paper.

\begin{itemize}
  \item a field $k$;
  \item $f\in k(x)$ a post-critically finite rational function defined over $k$
  and of degree $d\geq 2$;
\item $K_n= k(f^{-n}(t))$ for each $n\ge 0$, where $t$ is transcendental over $k$;
  \item $K_\infty = \bigcup_{n=0}^\infty K_n$;
 \item $k_n=\kbar \cap K_n $ for each $n\ge 0$, and $k_\infty = \kbar \cap K_\infty$;
\item $G_n = \Gal(K_n/k(t))$ for each $n\ge 0$;
\item $G_\infty=\varprojlim G_{n}\cong \Gal(K_\infty/k(t))$ is the limit of the $G_n$;
\item $\alpha \in \PP^1(k)$ an arbitrary point defined over $k$;
\item $K_{\alpha,n} =  k(f^{-n}(\alpha))$ for each $n\ge 0$,
and $K_{\alpha,\infty} = \bigcup_{n=1}^\infty K_{\alpha,n}$;
  \item $G_{\alpha,n} = \Gal(K_{\alpha,n}/k)$ for each $n\ge 0$;
  \item $G_{\alpha,\infty}=\varprojlim G_{\alpha,n}\cong \Gal(K_{\alpha,\infty}/k)$
  is the limit of the $G_{\alpha,n}$;
  \item $T^d_\infty$ is the infinite $d$-ary rooted tree.
  \item $T^d_n$ is the $d$-ary rooted tree with $n$ levels.
    % \item $k_{2^\infty}$ is $k$ with all $2^n$-th roots of unity
    %   adjoined for all $n$.
    \end{itemize}

Here, $f^n$ denotes the iterated composition $f\circ \cdots \circ f$, with $f^0(x)=x$
and $f^1(x)=f(x)$, and $f^{-n}(a)$ denotes the inverse image of $a$ under $f^n$.
We assume for each $n$ that the equations $f^n(x)-t=0$ and $f^n(x)-\alpha$ are separable,
so that each of the fields $K_n$ and $K_{\alpha,n}$ are indeed Galois extensions
of $K_0=k(t)$ and of $K_{\alpha,0}=k$, respectively.
By identifying the $d^n$ elements of $f^{-n}(t)$ or of $f^{-n}(\alpha)$,
counted with multiplicity, with the $d^n$ vertices at the $n$-th level of the tree $T^d_n$,
the Galois groups $G_n$ and $G_{\alpha,n}$ act on $T^d_n$.
Similarly, the Galois groups $G_\infty$ and $G_{\alpha,\infty}$ act on $T^d_\infty$.

Recall that we say a point $P\in\PP^1$ is \emph{preperiodic} under $f$
if there are integers $n>m\geq 0$ such that $f^n(P)=f^m(P)$,
and we say $f$ is \emph{post-critically finite}, or PCF, if all critical point of $f$ are preperiodic.

    \subsection{Overview of the problem}
    Note that for $\alpha \in k$, there is a simple way to view
    $G_{\alpha,\infty}$ as a subgroup of $G_\infty$ up to conjugacy
    whenever $\alpha$ is not strictly post-critical.
    (That is, whenever $\alpha$ does not equal $f^n(c)$ for any critical point $c$
    of $f$, and any integer $n\geq 1$.)
    Let $\fp$ be the prime corresponding to $\alpha$ in $\Spec k[t]$,
    and let $\fq_n$ be a prime lying over it in the integral closure
    of $\Spec k[t]$ in $K_{\alpha,n}$.  Then the decomposition group
    of $\fq_n$ over $\fp$ is isomorphic to $G_{\alpha,n}$.  Choosing
    primes
    \[ \fq_1 \subseteq \fq_2 \cdots \subseteq \fq_n \subseteq \cdots \]
    gives an embedding of $G_{\alpha,\infty}$ into $G_\infty$,
    well-defined up to conjugacy.

    The following has become a standard question in the area (see
    \cite{BostonJonesArboreal, BostonJonesImage, RafeArborealSurvey,
      BDGHT2}).   

    \begin{question}\label{q}
     Let $f$ be a rational function defined over a number field
     $k$.  Let $\alpha \in k$ be a point that is not strictly post-critical for
     $f$ and is not fixed by any rational function that commutes
     with any iterate of $f$.
     Then do we have $[G_\infty: G_{\alpha,\infty}] < \infty$?
    \end{question}

    Question \ref{q} is known to have a positive answer for non-PCF
    quadratic and cubic polynomials (see \cite{BDGHT1, BDGHT2, BT2, JKL})
    if one assumes the $abc$ conjecture along with a conjecture about
    irreducibility of iterates of rational functions due to Jones and
    Levy \cite{RafeAlon}.  Jones and Manes have proved similar results
    for special families of quadratic rational functions \cite{JonesManes}.

    Unconditionally, less is known about Question~\ref{q}, but
    partial results are known in some cases.
    Over function fields, Odoni \cite{OdoniIterates,OdoniWreathProducts}
    (see also \cite{Juul}) has shown that for generic pairs $(f,\alpha)$, $G_{\infty}$
    is all of $\Aut(T^d_\infty)$.
    When $d=2$, Odoni further proved that the example
    of $(x^2-x+1,0)$ over $\QQ$ \cite{OdoniExample} \cite{OdoniExample},
    satisfies $G_{\alpha,\infty}=G_\infty=\Aut(T^d_\infty)$.
    Stoll \cite{Stoll92} later extended Odoni's construction
    to infinitely many quadratic polynomials over $\QQ$.
    Looper \cite{Looper} produced infinitely many such examples in any
    prime degree $d=p$, later generalized to all degrees $d\geq 2$ in
    \cite{BenJuul,Kadets,Specter}.

    When $G_\infty\subsetneq\Aut(T^d_\infty)$,
    for example when $f$ is PCF, there had been fewer examples
    for which the answer to Question~\ref{q} was known.
    As shown in \cite{ABCCF}, the answer is yes for nearly all $\alpha\in k$ for
    the polynomial $f(x)=x^2-1$,
    using a Hilbert irreducibility argument.
   A somewhat similar result for a specific PCF cubic has been shown in \cite{BFHJY},
   and then generalized to normalized Belyi maps in \cite{BEK}.  We
   note however that the results there only yield infinitely many
   $\alpha$ such that $G_\infty = G_{\alpha, \infty}$, rather than
   nearly all $\alpha\in k$.
   
    Here we show that Question \ref{q} has a positive answer for any
    PCF quadratic rational function and for nearly all $\alpha \in k$.

    \begin{thm}\label{hilbert-plus}
     Let $f$ be a PCF quadratic rational function defined over a
     number field $k$.  Then for all $\alpha \in k$ outside of a thin set,
     we have $G_\infty = G_{\alpha,\infty}$.  
    \end{thm}

    Theorem \ref{hilbert-plus} follows from combining the result below
    with the Hilbert irreducibility theorem.

    \begin{thm}\label{more-general}
  Let $k$ be a number field and let $f \in k(x)$ be a PCF rational
  function such that $\Gal(K_1/k_1(t))$ is a $p$-group.  Then
  there is an integer $m\geq 1$ (depending on $f$ and $k$) such that
  $G_\infty = G_{\alpha,\infty}$ whenever $G_m = G_{\alpha,m}$.
 \end{thm}

 Note in particular that Theorem \ref{more-general} applies to any PCF quadratic
 rational function.

 We have a little bit more precision about the integer $m$ from the conclusion of Theorem~\ref{more-general} in the case of certain
 polynomials. 

 \begin{thm}\label{quadratic3}
   Let $f(x) = x^{p^n} + c$ be a PCF polynomial defined over a number
   field $k$.  Let $N$ be the size of the forward orbit of the
   critical point $0$.  Then there is a finite extension $k'$ of $k_1$
%   (see Section~\ref{subsec:tec}) 
   such that for any $\alpha\in k$, we have
   $G_{\alpha, \infty} = G_\infty$ if and only if
  \begin{equation}\label{k1}
    |\Gal(K_{\alpha, N} \cdot k' / k)| =|G_N| \cdot [k': k_1].
  \end{equation}
\end{thm}

\begin{remark}
Theorem \ref{quadratic3} is stated with an explicit description of
$k'$ in Theorem \ref{quadratic2}; more precisely, we have that $k'$ is the compositum of the finitely many extensions of $k$
of degree $p$ over $k_1$ contained in $k_\infty$.
\end{remark}

\subsection{Our strategy of proof}

The proofs of both Theorems~\ref{more-general}~and~\ref{quadratic3}
leverage properties of Frattini subgroups
(see Section \ref{more}).  Because $G_\infty$ is a $p$-group
in Theorems \ref{more-general} and \ref{quadratic3}, every maximal
closed subgroup of $G_\infty$ is normal and of index $p$ in $G_\infty$.
By the theory of Frattini subgroups, then,
Theorem~\ref{more-general} reduces to showing that $K_\infty$ contains
only finitely many Galois extensions of degree $p$ over $K$.
In Theorem~\ref{k} of Section~\ref{ext}, we show that for any number field $k$,
the resulting base field extension $k_\infty$ contains only finitely many extensions of
$k$ of bounded degree (with no conditions on $f$).
If we add the hypothesis that $f$ is PCF, then using standard facts
about fundamental groups, we prove in Lemma~\ref{geo} that
the field $K_\infty$ also contains only finitely many extensions of bounded degree
%over $\kbar(t)$.
over $k_\infty(t)$.
Combining Lemma~\ref{geo} with Theorem~\ref{k} gives
Theorem~\ref{more-general}, which we prove in Section~\ref{more}.
In Section~\ref{quadratic},
using more careful arguments about the exact number of extensions of
degree $p$ in $K_\infty$ (see Lemma~\ref{count1}),
we prove Theorem~\ref{quadratic2}, which is the more precise and explicit form
of Theorem~\ref{quadratic3}.

We note that our proofs do not involve deriving information about the
group $G_\infty$ itself (as is done in \cite{JKMT, PinkQuadratic,
  PinkQuadraticInfiniteOrbits, ABCCF, BFHJY}, for example).  In a
future paper, we plan to give a concrete presentation of the explicit
data about $G_\infty$ for PCF quadratic polynomials developed in
\cite{PinkQuadratic}.

% \begin{lemma}
%   Let $L$ be an extension of $k(t)$ that such that $L = K_\infty^H$
%   for $H$ a closed, maximal subgroup of $G_\infty$.  Then either $L =
%   \ell(t)$ for some finite extension $\ell$ of $k$ or $L \cap \kbar = k$.   
% \end{lemma}
% \begin{proof}
%   Since $G/H$ is simple, $L$ cannot contain any Galois extensions of
%   $k(t)$ other than $k(t)$ and $L$ itself.  Let $\ell = L \cap
%   \kbar$.  Then $\ell(t)$ is a Galois extension of $k(t)$ contained in
%   $L$.  Thus $\ell(t)$ is either $L$ or $k(t)$.  
% \end{proof}

\section{Extensions of the base field}\label{ext}

In this section, we collect some information about the fields $k_n$.
We begin with the following dynamical analog of a standard result
\cite[Proposition~VII.4.1]{AEC}
on the ramification of Tate modules of elliptic curves.

%\begin{lemma}\label{ram}
%  Let $R$ be a discrete valuation ring with maximal ideal $\fp$,
%  and let $k$ be the field of fractions of $R$.  Let $g(x)\in R[t][x]$
%  be a polynomial with coefficients in $R[t]$, where $t$ is transcendental over $k$,
%  and let ${\bar g}(x)$ denote its image in $(R/\fp)[t][x]$.
%  Suppose that $\deg {\bar g} = \deg g$ and that ${\bar g}$ is separable
%  and irreducible over $(R/\fp)(t)$.  
%  Let $L$ be a splitting field of the polynomial $g(x)$ over $k(t)$.
%  Then $\fp$ does not ramify in $L \cap \kbar$.
%\end{lemma}

\begin{lemma}\label{ram}
  Let $R$ be a discrete valuation ring with maximal ideal $\fp$,
  and let $k$ be the field of fractions of $R$.  Let $g(x)\in R[x]$
  be a nonconstant polynomial with coefficients in $R$,
  and let ${\bar g}(x)$ denote its image in $(R/\fp)[x]$.
  Suppose that $\deg {\bar g} = \deg g$ and that ${\bar g}$ is separable
%  and irreducible over $R/\fp$.  
  over $R/\fp$.  
  Let $L$ be a splitting field of the polynomial $g(x)-t\in R[t][x]$ over $k(t)$,
  where $t$ is transcendental over $k$.
  Then $\fp$ does not ramify in $L \cap \kbar$.
\end{lemma}
\begin{proof}
  We may suppose that $R$ is $\fp$-adically complete since $\fp$ will
  ramify in $L \cap \kbar$ if and only if the completion of $R$ at
  $\fp$ ramifies in the compositum of $L \cap \kbar$ with the
  completion of $k$ at $\fp$.  We may also assume that $L$ does not
  contain any algebraic extensions of $k$ that are unramified at
  $\fp$, after replacing $R$ with its integral closure in the maximal
  unramified algebraic extension of $k$ at $\fp$
  %in $\ell$.
  in $L$.
  
  Let $\fq = \fp R[t]$.  Then the localization $R[t]_\fq$ is a
  discrete valuation ring with maximal ideal $\fq R[t]_{\fq}$.  We
  will begin by showing that the prime $\fq R[t]_{\fq}$ does not ramify in $L$.

  Let $M$ be an extension of $k(t)$ generated by a root of $g(x) - t$,
  and let $B$ be the integral closure of $R[t]_{\fq}$ in $M$.
%  Let ${\bar g}$ be the image of $g$ in $(R/\fp)[x]$.
  Then ${\bar g}(x) - t\in (R/\fp)[t][x]$
  is separable and irreducible over $(R/\fp)(t)$ and is of degree $\deg g$.
  Hence, there
  is a single prime $\fm$ of $B$ lying over $\fq R[t]_{\fq}$, and
  \[ \big[B/\fm: R[t]_{\fq} /(\fq R[t]_{\fq}) \big] = \deg g = [M:k(t)]. \]
  Thus, $\fq R[t]_{\fq}$ does
  not ramify in $M$. Because $L$ is the compositum of $\deg g$ copies of $M$
  (one for each root of $g(x) - t$), we see that $\fq R[t]_\fq$ does not ramify in $L$,
  as we claimed above.

  Let $\ell = L \cap \kbar$.  Then we have a containment of fields
  $k(t) \subseteq k(t) \cdot \ell \subseteq L$.
  We must show that $\fp$ does not ramify in $\ell$;
  in fact, we will show that $\ell=k$.
%  which is equivalent to showing that $\ell = k$, since $\ell$ contains
%  no unramified extensions of $k$.

  Let $R'$ be the integral closure of $R$ in
  $\ell$.  Since $R$ is complete, the ring $R'$ is a discrete
  valuation ring with maximal ideal $\fm$, and we have $\fp R' = \fm^e$ for
  some integer $e$.  Writing $\fm =  \alpha\cdot R'$
  (for a suitable element $\alpha\in R'$), we see that $R' = R[\alpha]$,
  since the only elements of $k[\alpha]$ with $\fm$-adic absolute
  value less than or equal to $1$ are those are in $R[\alpha]$.
  Hence, the integral closure of $R[t]_\fq$ in $L$ is
  $R'[t]_\fq = R[t]_\fq[\alpha]$. 
%  Let $h(x)$ be the monic irreducible polynomial
%  in $R[x]$ such that $h(\alpha) = 0$.
  Let $h(x)\in R[x]$ be the minimal polynomial of $\alpha$ over $k$.
  Let ${\bar h}$ be the image of
  $h$ in $(R/\fp)[x]$.  Then ${\bar h} = (x-\bar{\beta})^e$ for some
  $\beta \in R$, since $\fm^e = \fp$, by the Dedekind-Kummer theorem
  (see, for example, \cite[Proposition~I.8.1]{Neukirch})
  applied to the extension $R[\alpha]$ of $R$.
  
  The polynomial $h$ remains irreducible over $k(t)$,
  and thus applying the Dedekind-Kummer theorem to the
  extension $R[t]_\fq[\alpha]$ of $R_\fq$, we see that
  $\fq R[t]_\fq[\alpha]$ must also have the form $\fn^e$ for some
  maximal ideal $\fn$ in $R[t]_\fq[\alpha]$.  Since $R[t]_\fq[\alpha]$ does
  not ramify in $L$, we must have $e = 1$, and hence $\ell = k$.
\end{proof}

\begin{defin}
  Let $k$ be a field, and let $f: \PP_k^1 \lra \PP_k^1$, written as
  $[P(x,y) : Q(x,y)]$ with $P,Q$ each homogenous of the same degree $d\geq 1$ in
  $k[x,y]$.  Let $R$ be a Dedekind domain with field of fractions
  $k$.  We say that $f$ has (\emph{explicit}) {\em good reduction} at a prime $\fp$ of $R$
  if the coefficients of $P$ and $Q$ are in $R_\fp$, the reductions 
  $P_\fp, Q_\fp\in (R/\fp)[x,y]$ of $P$ and $Q$ at $\fp$
  have no common roots in the algebraic closure of
  $R/\fp$, and $\max(\deg P_\fp, \deg Q_\fp) =d$.    We say
  that $f$ has (\emph{explicit}) {\em good separable reduction} at $\fp$ if in addition
  the map over $R/\fp$ sending $[x:y]$ to $[P_\fp(x,y) : Q_\fp(x,y)]$
  is separable.   
\end{defin}

Note that if $f$ has good separable reduction at a prime $\fp$, then so does $f^n$ for
any $n\geq 0$.

\begin{lemma}\label{ram2}
  Suppose that $k$ is the field of fractions of a Dedekind domain $R$
  and that the PCF rational function $f$ is separable.  Then there are at most finitely
  many primes $\fp$ of $R$ that ramify in $k_\infty$.
\end{lemma}
\begin{proof}
Since $f$ is separable, the set $S$ of primes at which $f$ fails to have
good separable reduction is a finite set.
Hence, for all $n\geq 0$ and all primes $\fp$ of $R$ outside of $S$,
the function $f^n$ has good separable reduction.
By Lemma~\ref{ram}, it follows that any such $\fp$ is unramified in $K_n \cap \kbar$,
for all $n\geq 0$. Therefore, any prime $\fp$ outside the finite set $S$
is unramified in $k_{\infty}$.
\end{proof}

We will also use the following result, which is proved in
\cite[Theorem~II.2.13]{Neukirch}.

\begin{lemma}\label{N}
  Let $D\in\mathbb{N}$ and let $S$ be a finite set of places of a number field $L$.
  There are only
  finitely many extensions $L'$ of $L$ unramified outside of $S$
  such that $[L:L'] \leq D$.
\end{lemma}

Lemmas~\ref{ram2}~and~\ref{N} immediately yield the following useful result.

\begin{thm}\label{k}
Let $k$ be a number field.  Then for any $D\geq 1$,
the field $k_\infty$ contains only finitely many intermediate fields $k'$
with $[k':k]\leq D$.
\end{thm}
\begin{proof}
  By Lemma \ref{ram2}, there at most finitely many primes in
  $k$ that ramify in $k_\infty$. Let $S$ be the set of all such primes
  together with the archimedean places of $k$. Applying Lemma \ref{N} to $S$
  yields the desired result.
\end{proof}

We will make use also of the following simple lemmas from commutative algebra.

\begin{lemma}\label{gen-close}
  Let $A$ be a complete discrete valuation ring with maximal ideal
  $\fp$ and field of fractions $k$, let $B$ be its integral closure in
  a finite, separable, unramified extension $M$ of $k$, and let $\fm$
  be the maximal ideal of $B$.  Suppose that there is some $\alpha\in B$
  such that $M = k(\alpha)$ and the minimal polynomial
  $g(x) \in A[x]$ of $\alpha$
  has the property that the reduction ${\bar g}(x)$ of $g$ in $A/\fp[x]$
  does not have repeated roots.  Then $B = A[\alpha]$.
\end{lemma}
\begin{proof}
  The ring $A[x]$ is Noetherian by the Hilbert basis theorem, and hence
  so is its quotient $A[\alpha]$. Moreover, $A[\alpha]$ has dimension~one
  because $A$ is integrally closed and of dimension~one, by the going-up
  and going-down theorems.

  By a variant of Hensel's Lemma (see, for example,
  \cite[Proposition~3.3.4.1]{BGR}), the reduced polynomial
  ${\bar g}$ is irreducible in $A/\fp[x]$, because it does not have repeated roots,
  and because the original polynomial $g$ is irreducible over $k$.
  We have
  $A[\alpha] / \fp A[\alpha] \cong (A/\fp)[x] / {\bar g(x)}$, which is
  a field, since ${\bar g}$ is irreducible.
  Therefore $\fp A[\alpha]$ is a maximal ideal in the ring $A[\alpha]$;
  since it is generated by $\fp$, it must be the unique maximal ideal.
  Writing $\fp=\pi A$ for some uniformizer $\pi\in A$, then
  this unique maximal ideal in $A[\alpha]$ is the principal ideal
  $A[\alpha] \pi$.

  As a Noetherian local domain of dimension~one whose maximal ideal is principal,
  $A[\alpha]$ must be a discrete valuation ring and hence integrally closed.
  (See, for example, \cite[Proposition~9.2]{AM}.)
  Since $A[\alpha]\subseteq B$ has the same field of fractions $M$
  as $B$ does, it follows that $A[\alpha]=B$.
\end{proof}

\begin{lemma}\label{comp}
Let $A$ be a discrete valuation ring with maximal ideal $\fp$ and
with field of fractions $k$.
Let $L_1$ and $L_2$ be finite separable extensions of $k$,
both contained in a common algebraic closure of $k$.
For each $i=1,2$, let $R_i$ be the integral closure of $A$ in $L_i$,
and let $\fq_i$ be a maximal ideal of $R_i$.
Let $B$ be the integral closure of $A$ in $L_1 \cdot L_2$,
and let $\fm$ be a maximal ideal of $B$ lying above both $\fq_1$ and $\fq_2$.
Suppose that $\fq_1$  does not ramify over $\fp$
and that $R/\fq_1$ is separable over $A/\fp$.
Then $B/\fm = R_1 / \fq_1 \cdot R_2 / \fq_2$.
\end{lemma}

In the final sentence of Lemma~\ref{comp}, note that the fields
$ R_1 / \fq_1$ and $R_2 / \fq_2$ both embed naturally into $B/\fm$ under the
inclusions of $R_1$ and $R_2$ into $B$. Thus, the conclusion is that
the compositum of these two quotient field is the whole field $B/\fm$.

\begin{proof}
By passing to the completion of $B$ at $\fm$,
we may assume that the rings $A$, $R_1$, $R_2$, and $B$ are complete
with respect to their (unique) maximal ideals.
%It suffices to prove this after passing to the completion of $B$ at $\fm$.
Choose $\alpha \in R_1$ such that the image of $\alpha$
  in $R_1/\fq_1$ generates $R_1/\fq_1$ over $A/\fp$; such an $\alpha$
  exists by the primitive element theorem,
  since $R_1/\fq_1$ is separable over $A/\fp$. 
  Let $g(x)\in A[x]$ be the minimal polynomial of $\alpha$ over $k$.
  Then ${\bar g}$ must be irreducible over $A/\fp$, because
  $\fq_1$ is unramified over $\fp$, and hence $[L_1:k] = [R_1/\fq_1: A / \fp]$.
  Furthermore, ${\bar g}$ must be separable since $R_1/ \fq_1$ is
  separable over $A/\fp$.
  Thus, $R_1 = A[\alpha]$, by Lemma~\ref{gen-close}.
  
  Let $h\in R_2[x]$ be the minimal polynomial of $\alpha$ over $R_2$.
  Then $h$ divides $g$ in $R_2[x]$, so the reduction ${\bar h}$ of $h$ in $R_2/\fq_2$
  divides ${\bar g}$.  In particular, ${\bar h}$ is also separable.
  Thus, by Lemma~\ref{gen-close} again, we have
  $B = R_2[\alpha]$,
  which is the subring $R_1\cdot R_2$ of $B$ generated by $R_1$ and $R_2$.
  It follows immediately that $B/\fm = R_1 / \fq_1 \cdot R_2 / \fq_2$.
\end{proof}

\begin{prop}\label{kn}
  Let $f$ be a separable rational function defined over a field $k$,
%  Let $f$ be a separable rational function defined over a number field $k$,
  let $n$ be a positive integer, and let $\alpha \in k$ be any point
  such that $f^n$ does not ramify over $\alpha$.
  (That is, there are no critical points $c$ of $f^n$ for which $f^n(c)=\alpha$.)
  Then $K_{\alpha,n}=k(f^{-n}(\alpha))$ contains $k_n=\kbar\cap K_n$. 
\end{prop}
\begin{proof}
  For each $\beta_i\in K_n$ such that $f^n(\beta_i) = t$,
  let $L_{i} = k(\beta_i)$; the field $K_n$ is the compositum of these $L_i$.
  Let $A$ be the local ring for the ideal $\fp=(t-\alpha)$ in $k[t]$, let $R_i$ be the
  integral closure of $A$ in $L_i$ for each $i$, let $B$ be the
  integral closure of $A$ in $K_n$, let $\fm$ be a maximal ideal in
  $B$, and let $\fq_i = R_i \cap \fm$ for each $i$.  Note that none of
  the primes $\fq_i$ ramify over $\fp$, since $\beta_i$ is not a critical point of $f^n$.
%  does not ramify over $\alpha$.
  The field $K_{\alpha,n}=k(f^{-n}(\alpha))$ contains the
  compositum of the fields $R_i/\fq_i$, which is equal to $B/\fm$ by
  Lemma \ref{comp}.  Since $k_n$ is contained in $B$, we see that
  $B/\fm$ contains $k_n$, so $K_{\alpha,n}$ must also contain
  $k_n$, as desired.
\end{proof}

We will use the following standard lemma from Galois theory throughout
the paper; see \cite[Theorem~VI.1.12]{Lang}. We include a proof for completeness.  
%NOTE: Lang includes sketch of idea why it's true
%without assuming finite extension

\begin{lemma}
\label{isomgal}
%Let $K$ and $L$ be finite, separable field extensions of a field $F$,
Let $K$ and $L$ be separable field extensions of a field $F$,
contained in the same algebraic closure of $F$.  Suppose that $K$ is
normal over $F$.  Then the natural restriction map $r: \Gal(K \cdot L / L)\lra 
\Gal(K/K \cap L)$ is an isomorphism.    
\end{lemma}
\begin{proof}
It suffices to prove the statement when $K \cap L = F$.
Clearly $r$ is a homomorphism. Any $\sigma\in\ker(r)$ acts trivially on
both $L$ and $K$ and is thus $1_{KL}$.

Let $H\subseteq\Gal(K/F)$ be the image of $r$.
We claim that the fixed field $K^H$ is $F$.
Clearly any $\gamma\in F$ is fixed by every $\sigma\in H\subseteq\Gal(K/F)$.
Conversely, any $\gamma\in K^H$ is fixed by every
$\sigma\in\Gal(K\cdot L /L)$ and hence lies in $L$.
Therefore, $\gamma\in K\cap L = F$, proving our claim.
By the Galois correspondence, it follows that $H=\Gal(K/F)$,
and hence that $r$ is surjective.
%Thus, it suffices to prove that $[K \cdot L: L] = [K:F]$.
%Write $K = F(\gamma)$, and let $g \in L[x]$ be the minimal polynomial
%of $\gamma$ over $L$.
%Because $K$ is Galois, the coefficients of $g$ are in $K$ and thus in $L
%\cap K = F$. Hence the degree of
%$[L(\gamma): L] = [F(\gamma): F] = L$, so $[K \cdot L: L] = [K:F]$.  
\end{proof}

\begin{prop}\label{any}
Let $\ell$ an algebraic extension of a field $k$.
%Suppose $k$ is a number field, and let $\ell$ be an algebraic extension of $k$.
If the embedding induced by specializing $t$ to $\alpha$ gives an isomorphism
\begin{equation}\label{gg}
  \Gal(\ell \cdot K_{\alpha, \infty}/ \ell) \cong \Gal(\ell \cdot K_\infty / \ell(t)),
\end{equation}
then $G_{\alpha, \infty} \cong G_\infty$.  
\end{prop}
\begin{proof}
If $\alpha$ were post-critical (i.e., if $\alpha=f^n(c)$ for some critical point $c$ of $f$
and some $n\geq 1$), then isomorphism~\eqref{gg} would fail.
After all, in that case,
the Galois group on the left would have to act in exactly the same way on
the two or more copies of $T^d_\infty$ rooted at the multiple copies of $c$
inside the main tree $T^d_\infty$ rooted at $\alpha$.
%we would not have $\Gal( K_{\alpha, \infty}/ \ell) = \Gal(K_\infty /\ell(t))$.
Thus, $\alpha$ must not be post-critical.
Therefore, by Proposition \ref{kn}, we must have
\begin{equation}
\label{eq:ellinclude}
%\ell \cap K_n =
\ell\cap k_n \subseteq \ell \cap K_{\alpha, n}
\quad \text{for each } n\ge 1 .
\end{equation}

By Lemma~\ref{isomgal}, we have
$\Gal ( \ell \cdot K_{\alpha, n}/ \ell)\cong \Gal(K_{\alpha,n} /\ell\cap K_{\alpha,n})$,
and hence
\begin{align}
\label{eq:Gansize}
|G_{\alpha,n}| &= [K_{\alpha,n}:k]
= [K_{\alpha,n}:\ell\cap K_{\alpha,n}] \cdot [\ell\cap K_{\alpha,n}:k] \notag \\
&= |\Gal( \ell \cdot K_{\alpha, n}/ \ell)|\cdot [\ell \cap K_{\alpha, n} : k] .
\end{align}

Meanwhile, we have $[ (\ell\cap k_n)(t) : k(t) ] = [\ell \cap k_n : k]$, since $t$
is transcendental over $k$.
We also have
\begin{equation}
\label{eq:ellkn}
(\ell\cap k_n)(t) = (\ell \cap K_n)(t) = \big( \ell(t) \big) \cap \big( K_n(t) \big)
= \ell(t) \cap K_n .
\end{equation}
By Lemma~\ref{isomgal} again, we have
$\Gal ( \ell \cdot K_{n}/ \ell(t) )\cong \Gal(K_{n} /\ell(t) \cap K_{n})$, and hence
\begin{align}
\label{eq:Gnsize}
|G_n| &= [K_n:k(t)] = [K_n : \ell(t) \cap K_n] [(\ell \cap k_n)(t) : k(t)] \notag \\
& =\big|\Gal \big( \ell \cdot K_{n}/ \ell(t) \big) \big| \cdot [(\ell \cap k_n)(t) : k(t)]
=  |\Gal( \ell \cdot K_{\alpha, n}/ \ell)|\cdot [\ell \cap k_n : k],
\end{align}
where the second equality is by equation~\eqref{eq:ellkn},
and the fourth is by hypothesis~\eqref{gg}.

Combining equations~\eqref{eq:Gansize} and~\eqref{eq:Gnsize}
with the fact that $\ell\cap k_n \subseteq \ell \cap K_{\alpha,n}$
from inclusion~\eqref{eq:ellinclude}, it follows that $|G_{\alpha,n}| \geq |G_n|$.
However, we also have $|G_{\alpha,n}| \leq |G_n|$ by construction,
whence $|G_{\alpha,n}| = |G_n|$ for all $n\ge 1$.
Therefore $G_{\alpha, \infty} \cong G_\infty$, as desired. 
\end{proof}

The converse of Proposition \ref{any} is false in general
(take $\ell = \kbar$, for example), but as our next lemma shows,
it does hold when $\ell \subseteq k_\infty$.

\begin{lemma}\label{con}
Let $\ell$ be an algebraic extension of $k$ contained in
$k_\infty$. If $G_{\alpha, \infty} \cong G_\infty$, then
\begin{equation}\label{ggg}
  \Gal(\ell \cdot K_{\alpha, \infty}/ \ell) \cong \Gal(\ell \cdot K_\infty /  \ell(t)).
\end{equation}
 \end{lemma}
 \begin{proof}
 As in the proof of Proposition~\ref{any}, we may assume that
 $\alpha$ is not post-critical, so that Proposition~\ref{kn} applies.
 Suppose that
 $|\Gal(\ell \cdot K_{\alpha, m}/ \ell)| < |\Gal(\ell \cdot K_m / \ell(t))|$
 for some $m\ge 1$.
 Then equation~\eqref{eq:Gansize} from the proof of Proposition~\ref{any}
 still holds, but the last equality in equation~\eqref{eq:Gnsize} becomes a strict inequality.
 By hypothesis, we have $G_{\alpha,m}\cong G_m$, and thus
 it follows that $[\ell \cap K_{\alpha,m}:k] > [\ell  \cap K_m: k]$.

 Let $\ell' = \ell \cap K_{\alpha,m}$, which is a subfield of $k_{\infty}$ by hypothesis.
 Then $\ell' \cdot K_m$ is contained in $K_{\infty}$; moreover, by the inequality
 at the end of the previous paragraph, it is a nontrivial extension of $K_m$.

 Let $A\subseteq k[t]$ be the local ring for $\fp=(t-\alpha)$,
 and let $B$ be the integral closure of $A$ in $\ell'\cdot K_m$.
 Because $G_{\alpha, m}$ is the decomposition group of $\fp$ in $K_m$
 but is equal to all of $G_m$ by hypothesis,
 and because $\ell'\cdot K_m$ is a base field extension of $K_m$,
 it follows that there is only a single prime $\fm$ of $B$ above $\fp$.
 Moreover, because we assumed $\alpha$ is not post-critical,
 the extension $K_m/k(t)$ is unramified at $\fp$, and hence
\begin{equation}
\label{eq:D1}
[k_\fm: k] = [\ell' \cdot K_m: k(t)], 
\end{equation}
where $k_\fm$ is the residue field $k_\fm = B/\fm$.

Setting $L_1=\ell'(t)$ and $L_2=K_m$,
both of which are finite separable extensions of $k(t)$,
and letting $R_i$ be the integral closure of $A$
in $L_i$ and $\fq_i=\fm\cap R_i$ for $i=1,2$,
Lemma \ref{comp} yields
\begin{equation}  
\label{eq:D2}
k_\fm = B/\fm = R_1/\fq_1 \cdot R_2/\fq_2 \cong
\ell' \cdot K_{\alpha,m} = K_{\alpha,m}.  
\end{equation}
But since $\ell'\cdot K_m$ is a nontrivial extension of $K_m$, we have
\[ [K_{\alpha,m}: k] = |G_{\alpha,m}| = |G_m| = [K_m: k(t)] <  [\ell' \cdot K_m: k(t)], \]
which  contradicts \eqref{eq:D1}~and~\eqref{eq:D2}.  So we must have
\[|\Gal(\ell \cdot K_{\alpha,m}/ \ell)| = |\Gal(\ell \cdot K_m / \ell(t))| 
\quad \text{for all }m\ge 1, \]
which implies that \eqref{ggg} holds.
\end{proof}
 
The following is an immediate consequence of Proposition~\ref{any}
and Lemma~\ref{con}.  
 
\begin{thm}\label{geo-does}

  Let $\ell$ be any extension of $k$ contained in $k_\infty$.   Then
  \[ \Gal( K_{\alpha, \infty}/ \ell) = \Gal(K_\infty /
    \ell(t))\]
  if and only if $G_{\alpha, \infty} = G_\infty$.  
\end{thm}

\section{Proof of Theorems~\ref{hilbert-plus} and~\ref{more-general}}\label{more}

\subsection{Additional technical results}
We need several more useful facts in order to prove
Theorem~\ref{more-general}, beginning with a brief discussion of Frattini subgroups.

\begin{defin}
The \textit{Frattini subgroup} of a profinite group $G$ is the intersection of
all closed maximal subgroups of $G$.  
\end{defin}

The following is well-known.  We provide a short proof for
completeness. 
\begin{lemma}\label{easy}
Let $G$ be a profinite group, let $H$ be a closed subgroup of
$G$, and let $F$ be the Frattini subgroup of $G$.  If $H$ intersects
every coset of $F$ in $G$ nontrivially, then $H = G$.
\end{lemma}
\begin{proof}
  Since $H$ intersects every coset of $F$ in $G$ nontrivially, it must
  intersect every coset of any closed maximal subgroup $M$ in $G$
  nontrivially as well.  This means that $H$ is not contained in any
  closed maximal subgroup of $G$, which means that $H$ is all of $G$. 
\end{proof}

\begin{lemma}\label{1}
Let $F$ be the Frattini subgroup of $G_\infty$.  Then the following
are equivalent:
\begin{itemize}
 
\item $G_\infty$ has only finitely many closed maximal subgroups;
  \item $K_\infty^F$ (the fixed field of $F$) is a finite extension of $k(t)$.
  \end{itemize}
\end{lemma}
\begin{proof}
For the forward implication, denote the finitely many closed maximal subgroups
of $G_\infty$ as $H_1, \dots, H_n$.
Then $F = \bigcap_{i=1}^n H_i$, and hence 
$K_\infty^F = K_\infty^{H_1} \cdots K_\infty^{H_n}$.
We have $[K_{\infty}^{H_i} : k(t)] = [G_{\infty}:H_i] < \infty$ for each $i$,
since $H_i$ is a maximal subgroup.
Hence, $K_\infty^F$ is also a finite extension of $k(t)$.

Conversely, if $[K_\infty^F:k(t)] < \infty$, then
$H=\Gal(K_\infty^F/k(t))$ is a finite group, and hence it contains only finitely
many subgroups. These subgroups are in one-to-one
correspondence with the closed subgroups of $G_\infty$ containing $F$.
Since $F$ is contained in every closed maximal subgroup of $G_{\infty}$,
it follows that $G_\infty$ has only finitely many closed maximal subgroups.
%(Note that $F$ itself is closed, since it is an intersection of closed subgroups.)
\end{proof}

\begin{lemma}\label{H}
Let $L$ be a Galois extension of $k(t)$ contained in $K_{\infty}$
and containing $K_\infty^F$.  If $H$ is a closed
subgroup of $G_\infty$ such that the restriction of $H$ to $L$ is all of
$\Gal(L/k(t))$, then $H = G_\infty$.
\end{lemma}
\begin{proof}
%Note that $F\subseteq H$, because $H$ is a closed subgroup of $G_{\infty}$
%and hence is either all of $G_{\infty}$ or else a closed maximal subgroup of $G_{\infty}$.
By hypothesis, the homomorphism from $H$ to $\Gal(L/k(t))$
given by restriction to $L$ is surjective.
Since $\Gal(K_\infty^F/k(t))$ is a quotient of $\Gal(L/k(t))$,
the restriction homomorphism from $H$ to $\Gal(K_\infty^F/k(t))$
is also surjective.
That is, the natural homomorphism $H\to G_\infty /F$ is surjective,
meaning that $H$ intersects every coset of $F$ nontrivially.
Hence we have $H = G_\infty$, by Lemma \ref{easy}.
\end{proof}

\begin{lemma}\label{almost}
  Suppose that $K_\infty^F$ is a finite extension of $k(t)$.  Then
  there is an integer $m\geq 1$ (depending only on $f$ and $k$)
  such that for any $\alpha\in k$ for which $G_m \cong G_{\alpha,m}$,
  we have $G_\infty \cong G_{\alpha,\infty}$.
\end{lemma}
\begin{proof}
  Since $K_\infty = \bigcup_{n=1}^\infty K_n$ and  $K_\infty^F$ has finite
  degree over $k(t)$, there exists $m\geq 1$ such that $K_m$ contains
   $K_\infty^F$.
   
   Given any $\alpha\in k$ for which $G_m \cong G_{\alpha,m}$,
   Let $H\subseteq G_{\infty}$ be the (closed) decomposition subgroup
   of $G_{\infty}$ for the prime $(t-\alpha)$ of $k[t]$, so that $H\cong G_{\alpha,\infty}$.
   Since the image of $G_{\alpha,m}$ in $G_m$ is the restriction of $H$
   to $K_m\supseteq K_\infty^F$,
   Lemma~\ref{H} shows that $G_{\infty}=H\cong G_{\alpha,\infty}$.
\end{proof}

\begin{lemma}\label{simple}
  Suppose that $G_1$ is a $p$-group.
%  Suppose that $\Gal(K_1/k(t))$ is a $p$-group.
  Then $G_\infty$ is a pro-$p$ group.
\end{lemma}
\begin{proof}
It is well known that $G_\infty$ is a subgroup of the infinite wreath product
of $G_1$ (see, for example, \cite[Lemma~3.3]{JKMT}).
Since this product is a pro-$p$ group, so is $G_{\infty}$.
%  Note that $G_n$ is a $p$-group for each $n\geq 1$. Indeed,
%  proceeding inductively, each group $G_{n+1}$ is a wreath product of $d$
%  copies of $G_n$ by $G_1$; since $G_n$ and $G_1$ are $p$-groups,
%  so is this wreath product.
%  Taking the inverse limit, $G_{\infty}$ is a pro-$p$ group.
  \end{proof}

The following result is a standard fact regarding pro-$p$ groups.

  \begin{lemma}\label{pro-p}
   Let $G$ be a pro-$p$ group.  Then every closed maximal subgroup of
   $G$ is normal of index $p$ in $G$.  
  \end{lemma}

The next statement  follows from standard facts regarding \'{e}tale fundamental groups.

\begin{lemma}\label{geo}
  Let $\ell$ be an algebraically closed field.  
  Let $S$ be a finite set of primes in the field $\ell(t)$, and let
  $\overline{\ell(t)}$ be an algebraic closure of $\ell(t)$.
  Let $p$ be a rational prime that is not equal to the characteristic of $\ell$.
%   if $\ell$ has positive characteristic.  
%  $p$ be a rational prime that is not equal to the characteristic of $k$.  Then
  Then $\overline{\ell(t)}$ contains exactly $\frac{p^{|S| -1} - 1}{p-1}$ degree $p$
%  normal extensions of $k(t)$ ramified only above at primes from  $S$.
  normal extensions of $\ell(t)$ that are unramified away from primes in
  $S$.

\end{lemma}

\begin{proof}
  By \cite[X, Cor. 2.12]{Gro} (see also \cite[Section 7.1]{Volk} for a
  discussion over $\CC$), the number of Galois extensions of $\ell(t)$
  of degree $p$ in $\overline{\ell(t)}$ is equal to the number of normal subgroups of index
  $p$ in a free group of rank $|S| - 1$.   

  There are $p^s$ homomorphisms from a free group $G$ with $s$
  generators to $\ZZ/p\ZZ$, since each generator may be mapped to any
  of the $p$ elements of $\ZZ/p\ZZ$.  Hence, there are $p^s-1$
  nontrivial such homomorphisms.  Now, for each normal subgroup $N$ of
  index $p$ in $G$ there are exactly $p-1$ homomorphisms with
  kernel $N$, each determined by the image of a fixed generator $aN$
  for $G/N$.  So we obtain exactly $\frac{p^s -1}{p-1}$ normal
  subgroups of index $p$, as desired.  
\end{proof}

\begin{remark} 
Lemma~\ref{geo}  is false if $p=\charact k$,
since for any monic polynomial $g\in k[t]$, the splitting field of the polynomial
$x^p-x+g(t)\in k(t)[x]$ is a new degree $p$ normal extension of $k(t)$
ramified only above the place at infinity from $k(t)$.  
\end{remark}

\begin{lemma}\label{K}
  Fix $D\geq 1$, and suppose that $k_\infty$ contains only
  finitely many subfields of degree at most $D$ over $k$.
  Suppose further
  that $K_\infty\cdot \kbar$ contains only
  finitely many subfields of degree at most $D$ over $\kbar(t)$.
  Then $K_\infty$ has only finitely many
  extensions of degree at most $D$ over $k(t)$.  
\end{lemma}
\begin{proof}
 For any field $L$ with $k(t)\subseteq L \subseteq K_{\infty}$ and $[L:k(t)]\leq D$,
 the field $L\cdot \kbar$ satisfies
 $\kbar(t)\subseteq L\cdot\kbar \subseteq K_{\infty}\cdot\kbar$
 and $[L\cdot\kbar:\kbar(t)]\leq D$.
 Thus, by the second assumption, it suffices to show that for any such field $L$,
 there are only finitely many other such fields $L'$ satisfying
 $L' \cdot \kbar = L \cdot \kbar$.  

 For any such $L,L'$, we have
  $L \subseteq L \cdot L' \subseteq L \cdot L' \cdot \kbar = L \cdot \kbar$, and
  $[L \cdot L' : L] \leq D$.
  Hence, $L \cdot L' = L \cdot \ell$ for some field
  $\ell \subseteq k_\infty$ with $[\ell: k] \leq D$.  There are only
  finitely many such subfields $\ell$ in $k_\infty$, by the first assumption.
  Finally, for each
  such $\ell$, the field $L \cdot \ell$ has only finitely many subfields
  $L'$ that contain $k(t)$. Thus, our proof is complete.
\end{proof}

% \begin{lemma}\label{K}
%   Suppose that $G_\infty$ is a pro-$p$ group.  Let $L_1, \dots, L_n$ be a
%   set of fields of degree $p$ over $k(t)$ in $K_\infty$ such that
%   $L_i \cap \kbar = k$ for all $i$.  Suppose furthermore that for any
%   field $M$ of degree $p$ over $k(t)$ in $K_\infty$ with
%   $M \cap \kbar = k$ we have $\kbar \cdot L_i = \kbar \cdot M$ for
%   some $L_i$.  Then for every degree $p$ extension $L'$ of $k(t)$ in
%   $K_\infty$, either $L= L_i$ for some $i$ or we have
%   $L' \subseteq L_i \cdot \ell$ for some $\ell susbeteq K_\infty$ such
%   that $[\ell: k] = p$.
% \end{lemma}
% \begin{proof}
%   We note that since $G_\infty$ is a $p$-group, the field $K_\infty$
%   contains no nontrivial extensions of $k(t)$ of degree less than $p$
%   by Lemma \ref{pro-p}.
%   Let $L'$ be a degree $p$  extension of $L'$.  If
%   $L' \cap \kbar \not= k$, then $L' = \ell_j$ for some $j$.
%   Otherwise, there is an $L_i$ such that
%   $L_i \cdot \kbar = \kbar \cdot L'$, which gives
%   $L_i \cdot L' \subseteq \kbar \cdot L_i$.  Thus, there is an
%   $\ell \subseteq \kbar$ such that $L' \subseteq \ell \cdot L$.  Let
%   $\ell = L_i \cdot L' \cap \kbar$.  Since
%   $[L_i \cdot L': L_i] \leq p$, we have $[\ell: k] \leq p$, so either
%   $L_i = L'$ or $[\ell: k] = p$.
% \end{proof}

\begin{lemma}\label{number}
 Let $p$ be a rational prime, let $k$ be a field of characteristic not equal to $p$,
 and let $f(x)$ be a post-critically
   finite polynomial with coefficients in $k$ such that $G_1$
%   $\Gal(K_1 /   k(t))$
   is a $p$-group.  If $k_\infty$ contains only finitely many
%   extensions of $k(t)$ of bounded degree,
   extensions of $k$ of degree~$p$,
   then $K_\infty^F$ is a finite extension of $k(t)$.  
\end{lemma}
\begin{proof}
By Lemma~\ref{simple}, $G_{\infty}=\Gal(K_{\infty}/ k(t))$ is a pro-$p$ group.
Hence, $\Gal(\kbar \cdot K_\infty /\kbar(t))$, which is isomorphic to a subgroup
of $G_{\infty}$ by the natural restriction homomorphism, is also a pro-$p$ group.

Any degree~$p$ extension of $\kbar(t)$ inside $\kbar \cdot K_{\infty}$
corresponds to a closed maximal subgroup of the pro-$p$ group
$\Gal(\kbar \cdot K_\infty /\kbar(t))$
and hence is normal by Lemma~\ref{pro-p}.
%Meanwhile, note that $K_\infty$ is unramified over $k(t)$ outside the places
%corresponding to the post-critical set of $f$.
%Hence, $\kbar \cdot K_\infty$ is also unramified over $\kbar(t)$
%outside the analogous post-critical places of $
Meanwhile, note that $\kbar \cdot K_\infty$ is unramified over $\kbar(t)$
outside the places corresponding to the post-critical set of $f$.
The set of such places is finite, since $f$ is PCF.
Hence, by Lemma~\ref{geo}, there are only finitely many extensions
of $\kbar(t)$ in $K_{\infty}$ that are of degree $p$.

Thus, by Lemma~\ref{K} and the hypotheses,
$K_{\infty}$ contains only finitely many extensions of degree $p$ over $k(t)$.
Applying Lemma~\ref{pro-p}, it follows that $G_{\infty}$ has only finitely many
closed maximal subgroups.
The desired conclusion is then immediate from Lemma~\ref{1}.
\end{proof}

\subsection{Proof of our first two main theorems}
We are now ready to prove Theorem~\ref{more-general}.

   \begin{thm}[Theorem \ref{more-general}]
  Let $k$ be a number field and let $f \in k(x)$ be a PCF rational
  function such that $\Gal(K_1/k_1(t))$ is a $p$-group.  Then
  there is an integer $m\geq 1$ (depending on $f$ and $k$) such that
  $G_\infty = G_{\alpha,\infty}$ whenever $G_m = G_{\alpha,m}$.
 \end{thm}

\begin{proof}
We may assume that $k_1 = k$, since Theorem \ref{geo-does} tells us
that $G_\infty = G_{\alpha,\infty}$ whenever
$\Gal(K_{\alpha,\infty}/ k_1) = \Gal(K_\infty/k_1(t))$.
By Theorem~\ref{k}, the field $k_\infty$ contains only finitely many
   extensions of degree $p$ over $k$, so by Lemma~\ref{number}, the field
   $K_\infty^F$ is a finite extension of $k(t)$.
   Lemma~\ref{almost} then gives the desired result.
\end{proof}

%\begin{thm}
% \label{thm:levelup}
%   Let $f$ be a rational PCF function of degree 2 over a number field.
%   Then there is an integer $m\geq 1$ (depending only on $f$ and $k$) such that
%   $G_\infty = G_{\alpha,\infty}$ whenever $G_m = G_{\alpha,m}$.
% \end{thm}

When $\deg f=2$, $G_1$ is a $2$-group. Thus,
Theorem~\ref{hilbert-plus} follows from Theorem~\ref{more-general}
because
the set of $\alpha$ in the number field $k$ for which the specialization of the finite
extension $K_m/k(t)$ to $K_{\alpha,m}/k$ fails to preserve the Galois group is a thin set.

\subsection{Other consequences} 

Recently, there has been a good deal of work on iterated Galois groups
over local fields (see \cite{Anderson, Berger, Ingram, Sing}, for
example).  Our next result is an analog of Theorem \ref{more-general}
for local fields.

   \begin{thm}\label{local}
     Let $k$ be a finite extension of $\QQ_q$ for some prime $q$, and
     let $f \in k(x)$ be a PCF rational function such that
     $\Gal(K_1/k_1(t))$ is a $p$-group.  Then there is an integer 
     $m\geq 1$ (depending on $f$ and $k$) such that
     $G_\infty = G_{\alpha,\infty}$ whenever $G_m = G_{\alpha,m}$.
   \end{thm}
   \begin{proof}
     As in the proof of Theorem \ref{more-general}, we may assume that
%     As in the proof of Theorem \ref{thm:levelup}, we may assume that
     $k_1 = k$.  A finite extension of $\QQ_q$ has only finitely many
     extensions of bounded degree, so $k_\infty$ contains only finitely
     many extensions of bounded degree over $k$. Lemmas~\ref{almost}
     and~\ref{number} then give the existence of the desired integer $m$.
   \end{proof}

We are also able to prove a result about specializations of iterated
Galois groups from number fields to finite fields.  
   
   \begin{thm}\label{specialize}
     Let $k$ be a number field with ring of integers $\fo_k$, and let
     $f \in k(x)$.  Suppose that $k = k_\infty$ and that
     $\Gal(K_1/k(t))$ is a $p$-group.  Then for all but finitely many
     primes $\fp$ of $\fo_k$, the infinite iterated Galois group
     $G_\infty$ for $f$ over $k(t)$
%   finitely many primes $\fp$ of $k$, the infinite iterated Galois
     is isomorphic to the infinite iterated Galois group for the reduction
     $\bar{f}_\fp \in \fo_k / \fp [x]$ over $\fo_k / \fp (t)$.
   \end{thm}
   \begin{proof}
%     Once again, by Theorem~\ref{k}, the field $k_\infty$ contains
%     only finitely many extensions of bounded degree $p$ over $k$, so applying
     By Lemma~\ref{number}, the
     field $K_\infty^F$ is a finite extension of $k(t)$.
     Therefore, as in the proof of Lemma~\ref{almost}, there is an integer $m\geq 1$
     such that $K_m$ contains $K_{\infty}^F$.

     Let $h(x)=f^m(x)-t\in k(t)[x]$, and for a prime $\fp$ of $\fo_k$,
     let $\bar{h}_\fp(x) = \bar{f}_\fp^m(x) -t\in (\fo_k / \fp) (t)[x]$.
     (This reduction makes sense for all but finitely many primes $\fp$ of $\fo_k$.)
     Then $K_m$ is the splitting field of $h$ over $k(t)$,
     and we may define $K'_{m,\fp}$ to be 
     the splitting field of $\bar{h}_\fp$ over $\fo_k / \fp (t)$.
     Further define $G'_{m,\fp}=\Gal(K'_{m,\fp} / (\fo_k / \fp) (t))$.
     
     By \cite[Proposition~4.1]{JKMT}, we have $G_m\cong G'_{m,\fp}$
     for all but finitely many primes $\fp$ of $\fo_k$.
     Meanwhile, the infinite iterated Galois group $G'_{\infty,\fp}$ for the reduction
     $\bar{f}_\fp$ over $\fo_k / \fp (t)$ is isomorphic to a closed subgroup $H_{\infty,\fp}$
     of $G_{\infty}$.
     (In particular, as a Galois group, $G'_{\infty,\fp}$ is an inverse limit of 
     finite groups; therefore its image $H_{\infty,\fp}$ is as well, and hence
     it is closed in the profinite topology on $G_{\infty}$.)
     Thus, because the restriction of $H_{\infty,\fp}$ to $K_m$ is all of $G_m$,
     and because $K_m \supseteq K_{\infty}^F$, we have $H_{\infty,\fp}= G_\infty$
     for all but finitely many $\fp$, by Lemma~\ref{H}.   
    \end{proof}

\section{Proof of Theorem \ref{quadratic3}}\label{quadratic}

 In this section, we give some more precise results in the case of PCF
 polynomials of the form $f(x) = x^{p^n} + c$ for $p$ a prime number
 and $n\geq 1$ an integer.

\subsection{Preliminary technical results} We start with a few useful lemmas.

 \begin{lemma}\label{first}
%   Let $p$ be a rational prime and $n\geq 1$.
   Let $f(x) = x^{p^n} + c$ be a PCF polynomial defined over a field
   $k$ of characteristic other than $p$, and let
   $N=|\{f^i(0) : i\geq 0\}|$ be the size of the forward orbit of the
   critical point $0$.  Let $\ell$ be any algebraic extension of
   $k_1$.  Then $\Gal(\ell \cdot K_N / \ell(t))$ is isomorphic to the
   full $N$-th iterated wreath product of the cyclic group $C_{p^n}$,
   and $(\ell \cdot K_N) \cap \kbar = \ell$. 
   Furthermore, if $k = k_1$, then
   $G_{\infty}$ is isomorphic to a subgroup of the infinite iterated
   wreath product of $C_{p^n}$.
 \end{lemma}
 
 \begin{proof}
  Since $\Gal(\ell \cdot K_N / \ell(t)) = \Gal(\ell \cdot k_1(f^{-N}(t))/ \ell(t))$,
%  when $\ell$ contains $k_1$,
  to prove the first statement, it suffices to
  show that  $\Gal(\ell \cdot K_N / \ell(t))$ is isomorphic to the
   full $N$-th iterated wreath product of the cyclic group $C_{p^n}$
   under the assumption that $k = k_1$.  Let $u\in f^{-1}(t)\subseteq K_1$. Then
   $f^{-1}(t)=\{\zeta^i u : 0\leq i\leq p^n-1\}$, where $\zeta$ is a
   primitive $p^n$-th root of unity. Thus, the assumption that $k=k_1$
   says precisely that $\zeta\in k$.  Since $f(x)-t$ is irreducible
   over $\kbar(t)$, it follows that  $\Gal(\ell \cdot K_1 / \ell(t))
   \cong C_{p^n}$. 

   The prime $(t-c)$ of $k(t)$ associated with $c=f(0)$ ramifies in $\ell \cdot K_1$ as
   $(u)^{p^n}$, and hence the ramification group of $(t-c)$
%   associated with $c=f(0)$
   must be the whole Galois group
   $\Gal(\ell \cdot K_1 / \ell(t))$, since the ramification index
   $p^n$ equals the order of $\Gal(\ell \cdot K_1 / \ell(t))\cong C_{p^n}$.
   Meanwhile, the subset $S=\{0\}$ of the critical points of
   $f$ has the property that for any $a\in S$, any critical point $b$
   of $f$ (i.e., any $b\in \{0,\infty\}$), and any $0\leq i,j\leq N$, we
   have $f^i(a)\neq f^j(b)$ unless $a=b=0$ and $i=j$.  Together, these
   are precisely the hypotheses of \cite[Theorem~3.1]{JKMT},
   which shows that
   $\Gal(\ell \cdot K_N / \ell(t))$ is isomorphic to the full
   $N$-th iterated wreath product of the cyclic group $C_{p^n}$.
   Since the same argument applied to $\kbar$ shows
   that $\Gal(\kbar \cdot K_N / \kbar(t))$ is the same
   iterated wreath product,
   Lemma~\ref{isomgal} shows that we must have
   $(\ell \cdot K_N) \cap \kbar(t) = \ell(t)$.
   Intersecting with $\kbar$, it follows that
   $(\ell \cdot K_N) \cap \kbar = \ell$.

  Finally, as noted in the proof of Lemma~\ref{simple},
  $G_\infty$ is a subgroup of the infinite iterated wreath product of
  $G_1\cong C_{p^n}$.
%(see, for example, \cite[Lemma~3.3]{JKMT}).
%When $k=k_1$, the group $G_1$ is
%isomorphic to $C_{p^n}$ by the previous paragraph, so $G_{\infty}$
%is isomorphic to a subgroup of the infinite iterated wreath product of $C_{p^n}$.
\end{proof}

\begin{remark}\label{rm}
Whether or not $k$ and $k_1$ coincide,
%  Note that when $k \not= k_1$, we still have
Lemma~\ref{first} also yields the equality $k_N = k_1$.
After all, we have $k_1 \subseteq k_N$ trivially,
and then choosing $\ell=k_1$ in Lemma~\ref{first} gives
$k_N \subseteq (k_1 \cdot K_N) \cap \kbar = k_1$.
 \end{remark}

%The next Lemma follows from standard results from group theory. 

 \begin{lemma}\label{count1}
%   Let $p$ be a rational prime and $n\geq 1$.
   Let $f(x) = x^{p^n} + c$ be a PCF polynomial defined over a field
   $k$ of characteristic other than $p$, and let $N$ be the size of
   the forward orbit of the critical point $0$.  Then for any
   algebraic extension $\ell$ of $k_1$, the field
%   $\ell \cdot K_N$ contains exactly $(p^N - 1)/(p-1)$ extensions
   $\ell \cdot K_N$ contains at least $(p^N - 1)/(p-1)$ extensions
   of degree $p$ over $\ell(t)$.
 \end{lemma}
 \begin{proof}
 % By Lemma~\ref{first}, $G_{\infty}$ is an inverse limit of finite groups
%  and is also a subgroup of an iterated wreath product of copies of $C_p^n$.
%  Hence, $G_{\infty}$ is a pro-$p$ group, and therefore, by Lemma~\ref{pro-p},
% every degree~$p$ extension of $k(t)$ in $K_{\infty}$ is normal.
% %In addition, the fields $K_N$ and $K_{\infty}$ are both unramified away from the
% In addition, $K_{\infty}$ is unramified away from the
%  set $S=\{\infty\} \cup \{ (t-f^i(0)) | i\geq 0\}$ of primes of $k(t)$
%  corresponding to the forward orbits of the critical points $\infty$ and $0$;
%  by hypothesis, we have $|S|=N+1$.
 % Therefore, by Lemma~\ref{geo},
% % $K_N$ contains at most $(p^N - 1)/(p-1)$ extensions of $k(t)$ of degree $p$.
%  $K_\infty$ contains at most $(p^N - 1)/(p-1)$ extensions
%  of $k(t)$ of degree $p$.
 
   By Lemma \ref{first}, the group $\Gal(\ell \cdot K_N / \ell(t))$ is
   isomorphic to the full $N$-th iterated wreath product of the cyclic
   group $C_{p^n}$.  The abelianization of
   $\Gal(\ell \cdot K_N/ \ell)$ is thus isomorphic to $C_{p^n}^N$, since
   the abelianization of any wreath product $A \wr B$ is isomorphic to
   the product of the abelianizations of $A$ and $B$ (see
   \cite[p.~215]{dH00}).
%   according to \cite{RafeArborealSurvey}).  
Therefore, it suffices to show that
 $C_{p^n}^N$ has at least $(p^N - 1)/(p-1)$ subgroups of index $p$.
   
 To see this, observe that the elements of $C_{p^n}^N$ of order $p$
  are precisely those of the form $(a_1 p^{n-1} , \ldots, a_N p^{n-1})$,
  where $a_1,\ldots,a_N\in\{0,\ldots, p-1\}$ are not all $0$.
  There are $p^N-1$ such elements, and each belongs to an equivalence class
  of size $p-1$ that generates the same subgroup of order $p$.
  Thus, there are indeed at least $(p^N-1)/(p-1)$ subgroups of index $p$ in $C_{p^n}^N$.    
 \end{proof}

 \begin{lemma}\label{count-final}
   Let $f(x) = x^{p^n} + c$ be a PCF polynomial defined over a field  $k$ of characteristic other than $p$,
   and let $N$ be the size of the forward orbit of the critical point 0.  
   Let $\ell$ be an algebraic extension of $k_1$ contained in $k_\infty$.  Then for every
   $L \subseteq  K_\infty$ with $[L:\ell(t)] = p$, we have
   $L \subseteq k_\infty \cdot K_N$.
 \end{lemma}
 \begin{proof}
By the second statement of Lemma~\ref{first}, $\Gal(\kbar\cdot K_\infty / \kbar(t))$
is a subgroup of an iterated wreath product of $C_{p^n}$ and hence is
a pro-$p$ group. Therefore, by Lemma~\ref{pro-p}, every extension of $\kbar(t)$
of degree $p$ that is contained in $\kbar\cdot K_\infty$ is normal.
In addition, $\kbar\cdot K_\infty$ is unramified away from the set
$S=\{\infty\}\cup\{ (t-f^i(0)) | i\geq 0\}$ of primes of $k(t)$ corresponding
to the forward orbits of the critical points $\infty$ and $0$.
By hypothesis, we have $|S|=N+1$, and hence by Lemma~\ref{geo},
$\kbar \cdot K_\infty$ contains at most $(p^N-1)/(p-1)$
extensions of degree $p$ over $\kbar(t)$.

On the other hand,
by Lemma~\ref{count1}, the field $\ell \cdot K_N$ contains
at least $(p^N-1)/(p-1)$ extensions of degree $p$ over $\ell(t)$.
Thus, defining $U_{\ell,N}$ to be the set of subfields $L$
of $\ell\cdot K_N$ satisfying $[L:\ell(t)]=p$,
and defining
$U_{\kbar,\infty}$ to be the set of subfields $M$
of $\kbar\cdot K_{\infty}$ satisfying $[M:\kbar(t)]=p$,
this means that
\begin{equation}\label{eq:Uineq}
\big| U_{\kbar,\infty} \big| \leq \frac{p^N-1}{p-1}  \leq \big| U_{\ell,N} \big| .
\end{equation}
We claim that the mapping
$\phi: U_{\ell,N} \to U_{\kbar,\infty}$ by $L\mapsto \kbar\cdot L$ is one-to-one.
%and hence that $| U_{\kbar,infty} | = | U_{\ell,N} |$, so that $\phi$ is bijective.

To prove the claim, first observe that any $L\in U_{\ell,N}$ is a geometric extension
of $\ell(t)$, since $L\subseteq\ell\cdot K_N$, and hence $L\cap\kbar=\ell$
by Lemma~\ref{first}.
Thus, we have
\[ [\kbar\cdot L : \kbar(t)]=[L:\ell(t)]=p, \]
so that $\phi(L)=\kbar\cdot L$ is indeed an element of $U_{\kbar,\infty}$.
In addition, if $L_1,L_2\in U_{\ell,N}$ satisfy $\phi(L_1)=\phi(L_2)$, then
\[ L_1\cdot L_2 \subseteq L_1\cdot L_2 \cdot \kbar =
L_1\cdot L_1 \cdot \kbar = \kbar \cdot L_1 . \]
At the same time, we also have $L_1\cdot L_2 \subseteq \ell \cdot K_N$, and hence
\[ L_1\cdot L_2 \subseteq (\kbar \cdot L_1) \cap (\ell\cdot K_N)
= \big(\kbar \cap (\ell\cdot K_N) \big) \cdot L_1 = \ell \cdot L_1 = L_1, \]
where the second equality is again by Lemma~\ref{first}.
Similarly, we also have $L_1\cdot L_2 \subseteq L_2$,
and hence $L_1=L_1\cdot L_2 = L_2$, proving the claim.

Given any $L$ as in the statement of the lemma, suppose first
that $L$ is not a geometric extension of $\ell(t)$.
Then $L\subseteq k_\infty (t) \subseteq k_{\infty} \cdot K_N$,
as desired.
Otherwise, the field $\kbar \cdot L$ is an extension
of $\kbar(t)$ of degree $p$ contained in $\kbar\cdot K_{\infty}$,
and hence $\kbar \cdot L \in U_{\kbar,\infty}$.
It follows from the claim and inequality~\eqref{eq:Uineq} that $\phi$ is bijective,
and hence there is some field $L'\in U_{\ell,N}$
so that $\kbar\cdot L = \kbar\cdot L'$.
As in the proof of the claim, it follows that
$L\cdot L' \subseteq \kbar \cdot L\cdot L' = \kbar\cdot L'$,
and hence $L\cdot L' \subseteq k_\infty \cdot L'$,
since $L\cdot L'\subseteq K_\infty$, which has constant field $k_\infty$.
Since $L'\subseteq \ell\cdot K_N$,
it follows that
\[ L\subseteq L\cdot L' \subseteq k_\infty \cdot L' 
\subseteq k_\infty \cdot (\ell \cdot K_N) \subseteq k_\infty \cdot K_N . \qedhere \]
\end{proof}

 \begin{thm}\label{geom}
  Let $f(x) = x^{p^n} + c$ be a PCF polynomial defined over a field
  $k$ of characteristic other than $p$.  
  Let $N$ be the size of the forward orbit of the critical point $0$.
  Then we have $G_{\alpha, \infty} = G_\infty$
%  \begin{equation}\label{g1}
%    G_{\alpha, \infty} = G_\infty
%  \end{equation}
  if and only if
 \begin{equation} \label{g2}
 \Gal(k_\infty \cdot K_{\alpha,N} / k_\infty) = \Gal(k_\infty \cdot K_N / k_\infty(t)).
\end{equation}
\end{thm}
\begin{proof}
Applying Theorem~\ref{geo-does} with $\ell=k_\infty$,
we may assume without loss that $k=k_\infty$.
The forward implication is then immediate by restriction
to $K_N$ and $K_{\alpha,N}$.

Conversely, suppose that equation~\eqref{g2} holds.
Then Lemma~\ref{count-final} with $\ell=k_\infty$
implies that every degree $p$ extension $L$ of
$k_\infty(t)$ in $K_\infty$ is contained in $k_\infty \cdot K_N$.
Hence, the fixed field $K_\infty^F$ is also
contained in $k_\infty \cdot K_N$, where $F$ is the Frattini
subgroup of $\Gal(K_\infty/k_\infty(t))$.

Recalling that we have assumed $k=k_\infty$,
let $L=k_\infty \cdot K_N = K_N$ and $H=G_{\alpha,\infty}$,
viewed as a (closed) subgroup of $G_{\infty}$.
The restriction of $H$ to $L$ is
\[ \Gal(k_\infty \cdot K_{\alpha,N} / k_\infty)
= \Gal(k_\infty \cdot K_N / k_\infty(t)) = \Gal(L/ k(t)),\]
by equation~\eqref{g2} and our assumption that $k=k_\infty$.
Therefore, by Lemma~\ref{H}, we have
$H=G_\infty$, as desired.
\end{proof}

\subsection{Proof of our second main theorem} We are ready to prove Theorem~\ref{quadratic3}, by obtaining a more precise verson of it, as stated below in Theorem~\ref{quadratic2}.

Because we work over a number field $k$, the field $k_\infty$ contains finitely
many extensions of $k$ of degree $p$ over $k_1$ by Lemma \ref{k}; let
$k'$ denote their compositum.
%Then for any Galois extension $\ell$ of
%$k_1$, we must have that $\ell$ intersects $k'$ in a non-trivial
%extension of $k_1$ since $\ell$ must contain a minimal non-trivial
%extension of $k_1$, which must be a degree $p$ extension of $k_1$,
%since $\Gal(K_\infty/k_1(t))$ is a pro-$p$ group. 
  
 \begin{thm}\label{quadratic2}
   Let $f(x) = x^{p^n} + c$ be a PCF polynomial defined over a number
   field $k$.  Let $N$ be the size of the forward orbit of the
   critical point $0$.  Let $k'$ be the compositum of the
   degree $p$ extensions of $k_1$ in $k_\infty$. 
   Then $G_{\alpha, \infty} = G_\infty$
   if and only if
  \begin{equation}\label{k11}
    |\Gal(k' \cdot K_{\alpha, N} / k)| =|G_N|\cdot [k': k_1]
  \end{equation}
\end{thm}

\begin{proof}
Suppose that $G_{\alpha, \infty} = G_\infty$.
Applying Theorem~\ref{geo-does}, we see that
$\Gal(K_\infty / k'(t)) = \Gal(K_{\alpha,\infty} / k')$,
and hence that $|\Gal(k' \cdot K_N / k'(t))| = |\Gal(k' \cdot K_{\alpha, N} / k')|$.
Therefore,
\begin{align*}
|\Gal(k' \cdot K_{\alpha, N} / k)| &= |\Gal(k' \cdot K_N / k'(t))| [k':k]
= \frac{|\Gal(K_N/k(t))|}{[k'(t) \cap K_N:k(t)]}  [k':k] \\
& = \frac{|G_N|}{[k' \cap K_N:k]}  [k':k] = |G_N|\cdot [k': k_1],
\end{align*}
where the second equality is by Lemma~\ref{isomgal},
%the third is because $[ k'(t) \cap K_N:k(t)] = [k'\cap K_N : k]$,
and the fourth is because
$k' \cap K_N = k_1$, by Remark~\ref{rm}.

Conversely, suppose that \eqref{k11} holds. We have
\begin{equation}\label{gn}
   [K_{\alpha, N}: k] = |G_{\alpha,N}| \leq |G_N|
\end{equation}
and also, by Lemma~\ref{isomgal} and the fact that $k_1\subseteq k'\cap K_{\alpha,N}$,
\begin{equation}\label{kg}
  [k' \cdot K_{\alpha, N}: K_{\alpha, N}] \leq [k': k_1].
\end{equation}
Therefore, by equation~\eqref{k11}, we have
\[ |G_N| \cdot [k':k_1]
= [k' \cdot K_{\alpha, N}: k]
= [k' \cdot K_{\alpha, N}: K_{\alpha, N}] \cdot [K_{\alpha, N}: k]
\leq [k':k_1] \cdot |G_N|, \]
so that we must have equality in both~\eqref{gn} and~\eqref{kg}.

Let $\ell=k_\infty\cap K_{\alpha,N}\supseteq k_1$. We claim that $\ell=k_1$.
To see this, note by Lemma~\ref{first} that $\Gal(K_N/k_1(t))$ is a $p$-group,
and hence so is the subgroup $H=\Gal(K_{\alpha,N}/k_1)$.
If $\ell\neq k_1$, then $\ell$ is a nontrivial extension of $k_1$ contained in $K_{\alpha,N}$,
and therefore $\Gal(K_{\alpha,N}/\ell)$ is contained in a maximal subgroup $H'$ of $H$,
which must have index $p$ in $H$. The fixed field of $H'$ is therefore an extension
$\ell'$ of $k_1$ of degree $p$ and contained in $k_{\infty}$,
so we have $\ell'\subseteq k'$ by definition of $k'$.
Therefore, by the same reasoning as in inequality~\eqref{kg},
\[ [k'\cdot K_{\alpha, N}: K_{\alpha, N}] \leq [k' : \ell'] = \frac{1}{p} [k' : k_1]
< [k':k_1] = [k'\cdot K_{\alpha, N}: K_{\alpha, N}], \]
where the final equality is because we showed above that~\eqref{kg} is an equality.
This contradiction proves our claim.

By the claim and Lemma~\ref{isomgal}, we have
\begin{align}
\label{eq:Kansize}
[k_\infty \cdot K_{\alpha, N}: k_\infty] &= [K_{\alpha, N}: k_1] 
= \frac{[K_{\alpha,N} : k]}{[k_1:k]} = \frac{|G_N|}{[k_1(t):k(t)]}
\notag \\
& = [K_N: k_1(t)] = [k_\infty \cdot K_N: k_\infty(t)].
\end{align}
Here, the third equality is because $[k_1(t):k(t)]=[k_1:k]$ and
because we showed~\eqref{gn} is an equality.
The fifth is by Lemma~\ref{isomgal} again, together with
the fact that $k_\infty \cap K_N = k_1$, by Remark~\ref{rm}.

Equation~\eqref{eq:Kansize} shows that the subgroup
$\Gal(k_{\infty}\cdot K_{\alpha,N} / k_\infty)$ of 
$\Gal(k_{\infty}\cdot K_N / k_\infty(t))$ is the whole group.
Applying Theorem \ref{geo-does} then gives
$G_{\alpha, \infty} = G_\infty$, as desired.  
\end{proof}

\section{Further questions}\label{further}

In light of Theorem \ref{more-general}, it is natural to ask whether the
Frattini subgroup of $G_\infty$ has finite index in $G_\infty$ for any
post-critically finite rational function defined over a number field.
Unfortunately, this is not the case, as \cite{BEK} gives examples
where $G_\infty$ is the infinite iterated wreath product of the
alternating group $A_d$.  It is easy to see that the infinite iterated
wreath product of any nontrivial group has infinitely many closed maximal
subgroups. It would be interesting to know wheter the Frattini subgroup of
$G_\infty$ has finite index in $G_\infty$ in the case where $f$ is a PCF
polynomial of the form $f(x) = x^m + c$ for $m$ not a prime power.

We would also like to explore a more general form of Odoni's
conjecture.  Recall that a field $k$ is said to be \emph{Hilbertian} if the
complement of any thin set in $k$ is infinite.  Odoni
\cite{OdoniIterates} conjectures that for any integer $d \geq 2$
and any Hilbertian field $k$ of characteristic 0, there is a polynomial
$f$ and an $\alpha \in k$ such that $G_{\alpha, \infty}$ is the full
automorphism group of $T^d_\infty$.  Dittman and Kadets \cite{DK}
have given counterexamples to this conjecture in every degree.

More generally, given a particular polynomial $f$ defined over a
number field $\ell$, one might ask whether it is true that for any
Hilbertian field $k$ containing $\ell$, there are infinitely many
$\alpha \in k$ such that $G_\infty = G_{\alpha, \infty}$.  (We note
here that the Galois groups are taken relative to $k$, not $\ell$, so
that $K_{n,\alpha} = k(f^{-n}(\alpha))$.)
For non-PCF quadratic polynomials, the the answer is no, using the
results of \cite{DK}. In fact, we have the following stronger result.

\begin{prop}
\label{DKquad}
Let $\ell$ be a number field, and let $f\in\ell[x]$ be a quadratic polynomial
that is not PCF. Then there is a Hilbertian field $k$ that is algebraic over $\ell$
such that for all $\alpha\in k$, the group $G_{\alpha,\infty}$
has infinite index in $G_\infty$.
\end{prop}

\begin{proof}
Since $\deg f=2$, there are exactly two critical points. Because $f$
is a non-PCF polynomial, the two critical orbits are disjoint, and one of those
orbits is infinite.
Therefore, by \cite[Theorem~4.8.1(a)]{PinkQuadraticInfiniteOrbits},
$G_\infty$ is the full automorphism group of $T^2_\infty$.
(See also \cite[Theorem~3.1]{JKMT}.)

According to \cite[Theorem~1.2]{DK}, there is 
some Hilbertian $k$, algebraic over $\ell$, such that for any quadratic polynomial
$g\in k[x]$, the arboreal Galois group for $g$ with base point $0$ over $k$ has
infinite index in $\Aut(T^2_\infty)$.
For any $\alpha\in k$, defining $g(x)=f(x+\alpha)-\alpha\in k[x]$,
it follows that $G_{\alpha,\infty}$ has infinite index in
$\Aut(T^2_\infty)=G_{\infty}$.
\end{proof}

On the other hand, 
combining the results of this paper with Pink's classification of
$k_\infty$ for PCF quadratic polynomials, rather than non-PCF,
the question above has a positive answer, as follows.

\begin{thm}
  Let $f(x) = x^2 + c$ be a PCF quadratic polynomial defined over a
  number field $\ell$.  Let $k$ be a Hilbertian field containing
  $\ell$.  Then there are infinitely many $\alpha \in k$ such that
  $G_\infty = G_{\alpha, \infty}$.
\end{thm}
\begin{proof}
  Pink \cite{PinkQuadratic} shows that $k_\infty$
%  is either a finite
%  extension of $k$ or a primitive $8$-th root of unity or is equal to
is contained in
  the field generated by all $2^n$-th roots of unity over $k$.
%  If $k_\infty$ is a finite extension of $k$, then it obviously contains
%  at most finitely many quadratic extensions of $k$. If $k_\infty$ is
%  equal to the field generated by all $2^n$-th roots of unity over $k$,
%  it contains either 1 or 2 quadratic extensions of $k$ (depending on
%  whether or not $k$ contains a primitive $8$-th root of unity).
  Thus, $k_\infty$ contains at most finitely many quadratic
  extensions of $k$.  Since $G_1$ is obviously a $2$-group, we have
  that $K_\infty^F$ is a finite extension of $k(t)$, by Lemma~\ref{number},
  where $F$ is the Frattini subgroup of $G_\infty$.
  Lemma \ref{almost} then implies that there is an $m\geq 1$ such that
  $G_\infty = G_{\alpha,\infty}$ whenever $G_m = G_{\alpha,m}$.  Since
  the set of $\alpha \in k$ such that $G_m = G_{\alpha,m}$ is the
  complement of a thin set in $k$ (see \cite[9.2, Proposition
  2]{Serre}), there are thus infinitely many $\alpha \in k$ such that
  $G_\infty = G_{\alpha,\infty}$ because $k$ is Hilbertian.
\end{proof}

It would be interesting to know whether there are other families of PCF
polynomials for which this variant of the Odoni conjecture holds.  

\medskip

\textbf{Acknowledgments}.
The first author gratefully acknowledges the support of NSF grant DMS-2101925. The second author was partially supported by an NSERC Discovery Grant. 
The authors thank
Vesselin Dimitrov,
Philipp Habegger,
Alexander Hulpke,
Jonathan Pakianathan,
Wayne Peng,
Harry Schmidt, 
Romyar Sharifi,
and David Zureick-Brown,
for helpful conversations.

\end{document}